\documentclass[12pt, reqno]{amsart}
\usepackage{amssymb,amsthm,amsmath}
\usepackage[numbers,sort&compress]{natbib}
\usepackage{amssymb,amsmath}
\usepackage{amsfonts}
\usepackage{mathrsfs}
\usepackage{latexsym}
\usepackage{amssymb}
\usepackage{amsthm}
\usepackage{indentfirst}
\date{December 2, 2015}

\let\oldsection\section
\renewcommand\section{\setcounter{equation}{0}\oldsection}

\newtheorem{theorem}{Theorem}[section]
\newtheorem{lemma}{Lemma}[section]
\newtheorem{proposition}{Proposition}[section]

\newcommand{\n}{\nabla}

\allowdisplaybreaks
\begin{document}

\title[Local well-posedness of liquid crystal system]{Local well-posedness of strong solutions to density-dependent
liquid crystal system}

\author{Huajun~Gong}
\address[Huajun Gong]{College of Information Engineering, Shenzhen University, Shenzhen 518060, PR China;}
\address{College of Mathematics and Statistics, Shenzhen University, Shenzhen 518060, PR China.}
\email{huajun84@hotmail.com}

\author{Jinkai~Li}
\address[Jinkai~Li]{Department of Computer Science and Applied Mathematics, Weizmann Institute of Science, Rehovot 76100, Israel.}
\email{jklimath@gmail.com}

\author{Chen~Xu}
\address[Chen~Xu]{Institute of Intelligent Computing, Shenzhen University, Shenzhen 518060, PR China.}
\email{xuchen@szu.edu.cn}



\begin{abstract}
In this paper, we study the Cauchy problem to the density-dependent
liquid crystal system in $\mathbb R^3$. We establish the local
existence and uniqueness of strong solutions to this system.
In order to overcome the difficulties caused by the high order coupling terms,
a biharmonic regularization of the system, as an auxiliary system, is introduced, and we
make full use of the intrinsic cancellation
properties between the high order coupling terms.
\end{abstract}

\maketitle

\allowdisplaybreaks

\section{Introduction}

In this paper, we study the following density-dependent liquid crystal system on $\mathbb R^3$:
\begin{equation}\label{el}\left\{
\begin{array}{l}
\partial_t\rho+u\cdot\nabla\rho=0,\\
\rho(\partial_tu+(u\cdot\n) u)+\n P=\Delta u -\n\cdot[\n d\odot \n d+(\Delta d+|\n d|^2d)\otimes d],\\
\text{div}\,u=0,\\
\partial_td+(u\cdot\n) d-(d\cdot\n) u=(\Delta d+|\n d|^2d)-(d^T A d)d,\\
|d|=1,
\end{array}
\right.
\end{equation}
where $\rho\in[0,\infty)$ is the density, $u=(u^1, u^2, u^3)$ is the velocity field, $P\in\mathbb R$ is the pressure, and $d=(d^1, d^2, d^3)\in S^2$, the unit sphere in $\mathbb R^3$, is the director field, $A=\frac12(\nabla u+(\nabla u)^T)$ and $\nabla d\odot\nabla d=(\partial_id\cdot\partial_jd)_{3\times3}$.

For the homogeneous case, i.e.\,the case that the density is a
constant, the
mathematical studies on the dynamical liquid crystal systems were
started by Lin--Liu \cite{LINLIU95,LINLIU00}, where they established
the global existence
of weak solutions, in both 2D and 3D, to the Ginzburg-Landau
approximation
of the liquid crystal system, see Cavaterra--Rocca--Wu \cite{CRW13}
and Sun--Liu \cite{SL09} for some generalizations to the general liquid
crystal systems, but still with Ginzburg-Landau approximation. Global
existence of weak solutions to the original liquid crystal systems in
2D, without the Ginzburg-Landau approximation, was established by
Lin--Lin--Wang \cite{LINLINWANG10}, Hong \cite{HONG11},
Hong--Xin \cite{HONGXIN12}, Huang--Lin--Wang \cite{HUANGLINWANG14}
and Wang--Wang \cite{WANGWANG14}, and in particular, it was shown
that global weak solutions to liquid crystal system in 2D have at most finite many singular times, while the uniqueness of weak solutions to liquid crystal
system
in 2D was proved by Lin--Wang \cite{LINWANG10}, Li--Titi--Xin \cite{LITITIXIN}, Wang--Wang--Zhang \cite{WANGWANGZHANG} and Xu--Zhang \cite{XUZHANG12}; global
existence (but without uniqueness) of weak solutions to the
liquid crystal system in 3D was recently
established by Lin--Wang \cite{LINWANG15}, under the assumption that
the initial director filed $d_0$ takes value from the upper half unite
sphere. If the initial data are suitably smooth, then the liquid crystal
system has a unique local strong solution, see Hong--Li--Xin \cite{HONGLIXIN14}, Li--Xin \cite{LIXIN1}, Wang--Wang \cite{WANGWANG14}, Wang--Zhang--Zhang \cite{WANGZHANGZHANG13}, Wu--Xu--Liu \cite{WUXULIU13} and Hieber--Nesensohn--Pr\"us--Schade \cite{HIEBER1}. Moreover, if the initial data is suitably small,
or the initial director filed satisfies some geometrical condition in
2D, then the local strong solution to the liquid crystal system can be extended to be a global one, see Hu--Wang \cite{HUWANG10}, Lei--Li--Zhang \cite{LEILIZHANG14}, Li--Wang \cite{LIWANG12}, Ma--Gong--Li \cite{MAGONGLI}, Gong--Huang--Liu--Liu \cite{GONGHUANGLIULIU15}
and Hieber--Nesensohn--Pr\"us--Schade \cite{HIEBER1}; remarkably, the recent work Huang--Lin--Liu--Wang \cite{HUANGLINLIUWANG}, concerning the finite time blow up of liquid crystal system in 3D, indicates that one can not generally expect the global existence of strong solutions to the liquid crystal system in 3D, without any further assumptions, beyond the necessary regularities, on the initial data. It is worth
to mention that some mathematical analysis concerning the global existence of weak solutions and local or global well-posedness of strong solutions of the non-isothermal liquid crystal systems were addressed by Hieber--Pr\"us \cite{HIEBER2}, Li--Xin \cite{LIXIN2}, Feireisl--Rocca--Schimperna \cite{FRS11} and Feireisl--Fr\'emond--Rocca--Schimperna \cite{FFRS12}.

In contrast with the homogeneous case, there are much less works concerning the mathematical analysis on the inhomogeneous liquid crystal system. As the counterparts of \cite{LINLIU95}, global existence of weak solutions
to the Ginzburg-Landau approximated density-dependent liquid crystal system was established by Liu--Zhang \cite{LIUZHANG09}, Xu--Tan \cite{XUTAN11} and Jiang--Tan \cite{JIANGTAN09}, while for the original liquid crystal system, i.e.\,the system without the Ginzburg-Landau approximation,
global existence of weak solutions was established only in 2D and under some geometrical assumptions on the initial director field, see
Li \cite{JKL14}, where the global existence of strong solutions was also established at the same time. Local well-posedness
of strong solutions to the three dimensional liquid crystal system was established by Wen--Ding \cite{WEN11}, while the corresponding
global well-posedness, under some smallness assumption on the initial data, was proved by Li--Wang \cite{LIWANG15} in the absence of vacuum, and by Li \cite{JKL15} and Ding--Huang--Xia \cite{DINGHUANGXIA} in the presence of vacuum.

Noticing that in all the existing works concerning the density-dependent liquid crystal systems, see, e.g., \cite{LIUZHANG09,XUTAN11,JIANGTAN09,JKL14,WEN11,LIWANG15,JKL15,DINGHUANGXIA} mentioned in the above paragraph, the systems under consideration are simplified versions of the general liquid crystal system. In the general liquid crystal system, there presents some
high order coupling terms which are as high as the leading terms, while these coupling terms were ignored in the works just mentioned.
Technically, the main difficulty resulted from the high order coupling terms is: on the one hand,
in order to get the energy estimates, one has to make full use of the intrinsic cancellation properties between these high order coupling terms, but on the other hand, if we solve the system by using the standard arguments such as the Galerkin approximation or linearization, it will destroy these intrinsic cancellation properties.

The aim of this paper is to study the density-dependent liquid crystal system, and we focus on such a system that has the high order
coupling terms. Precisely, we consider system (\ref{el}), and study the Cauchy problem of it. It should be noted that the arguments present in this paper apply equally to the general system but the calculations will be more involved. We complement system (\ref{el}) with the following boundary and initial conditions:
\begin{eqnarray}
  &&(u, d)\rightarrow(0,d_*),\quad\mbox{as }x\rightarrow\infty,\label{bc}\\
  &&(\rho, u, d)|_{t=0}=(\rho_0, u_0, d_0),\label{ic}
\end{eqnarray}
where $d_*$ is a constant unit vector.

Before stating the main result, let us introduce some notations which will be used throughout this paper. For $1\leq q\leq\infty$, we denote by $L^q(\mathbb R^3)$ the standard Lebesgue space. For positive integer $k$, $H^k(\mathbb R^3)$ is the standard Sobolev space $W^{k,2}(\mathbb R^3)$, while $D^{k,2}(\mathbb R^3)$ denotes the corresponding homogeneous Sobolev space
$$
D^{k,2}=\{\varphi\in L^6(\mathbb R^3)|\nabla^j\varphi\in L^2(\mathbb R^3), 1\leq i\leq k\}.
$$
For simplicity, we use the same notation $X$ for a Banach space $X$ and its $N$-product space $X^N$.

The main result of this paper is the following:

\begin{theorem}\label{th1}
Assume that the initial data $(\rho_0, u_0,d_0)$ satisfies
\begin{eqnarray*}
&&0\leq \rho_0\leq\bar{\rho},\quad \nabla\rho_0\in L^3(\mathbb{R}^3), \quad u_0\in D^{2,2}(\mathbb{R}^3), \\
&&\text{div}\, u_0=0, \quad d_0-d_*\in D^{3,2}(\mathbb{R}^3),\quad |d_0|=1,
\end{eqnarray*}
for a positive number $\bar\rho\in(0,\infty)$ and a constant unit vector $d_*$. Suppose further that the following compatibility condition holds
\begin{eqnarray*}
&&\Delta u_0-\n P_0-div(\n d_0\odot\n d_0+(\Delta d_0+|\n d_0|^2d_0)\otimes d_0)=\sqrt{\rho_0}g_0,
\end{eqnarray*}
for some $(P_0,g_0)\in H^1(\mathbb{R}^3)\times L^2(\mathbb{R}^3)$.

Then, there exists a positive time $T_*$, depending only on $\bar\rho$ and $\|\nabla u_0\|_{H^1}+\|\nabla d_0\|_{H^2}+\|g_0\|_{L^2}$, such that system (\ref{el}), subject to (\ref{bc})--(\ref{ic}), has a strong solution $(\rho, u,d)$, in $\mathbb R^2\times(0,T_*)$, which has the regularity properties
\begin{eqnarray*}
  &&(\nabla\rho, \partial_t\rho)\in L^\infty(0,T_*; L^3(\mathbb R^3)),\\
  &&u\in C([0,T_*]; D^{1,2}(\mathbb R^3))\cap L^\infty(0,T_*; D^{2,2}(\mathbb R^3))\cap L^2(0,T_*; D^{3,2}(\mathbb R^3)),\\
  &&\partial_tu\in L^2(0,T_*; D^{1,2}(\mathbb R^3)),\quad\partial_td\in
  L^\infty(0,T_*; H^1(\mathbb R^3))\cap L^2(0,T_*; H^2(\mathbb R^3)),\\
  &&d-d_*\in C([0,T_*]; D^{2,2}(\mathbb R^3))\cap L^\infty(0,T_*; D^{3,2}(\mathbb R^3))\cap L^2(0,T_*; D^{4,2}(\mathbb R^3)),
\end{eqnarray*}
and satisfies system (\ref{el}) pointwisely, a.e.~in $\mathbb R^3\times(0,T_*)$.

Moreover, if we assume in addition that
$$
\rho_0\in L^{\frac32}(\mathbb R^3),\quad \nabla\rho_0\in L^2(\mathbb R^3),
$$
then the strong solution $(\rho, u, d)$ stated above is unique.
\end{theorem}

As we mentioned before, the main difficulty for proving system (\ref{el}) is the presence of the high order coupling terms, which
are as high as the leading terms. To overcome this difficulty, we introduce a biharmonic regularized system, see Section \ref{sec2}, below, by adding a biharmonic term $\varepsilon\Delta^2u$ to the momentum equations. With the help of this biharmonic regularization term,
the coupling terms in the regularization system are no longer high order ones, and thus one can used the standard linearization argument to solve
the regularized system. Observing that this biharmonic regularization term does not destroy the intrinsic cancellation properties of
the coupling terms, one can obtain some appropriate $\varepsilon$-independent a priori estimates to the regularized solutions, in a uniform
time interval, and as a result, by passing $\varepsilon\rightarrow0$, one obtains the strong solutions to the original
system.

The arrangement of this paper is as follows: in the next section, Section \ref{sec2}, we prove the local existence of strong solutions
to a biharmonic regularized system, and perform the $\varepsilon$-independent a priori estimates in a uniform time interval; in the last section, Section \ref{sec3}, we prove Theorem \ref{th1}.

\section {A biharmonic regularized system}
\label{sec2}
In this section, we study  the following biharmonic regularized system of the original liquid crystal system (\ref{el}):
\begin{equation}\label{EL}
\left\{
\begin{array}{l}
\partial_t\rho+u\cdot\nabla\rho=0,\\
\rho (\partial_tu+(u\cdot\n) u)+\n P-\Delta u+\varepsilon\Delta^2u\\
\qquad=-\n\cdot[\n d\odot \n d+(\Delta d+|\n d|^2d)\otimes d],\\
\text{div}\,u=0,\\
\partial_td+(u\cdot\n )d-(d\cdot\n )u=\Delta d+|\n d|^2d-(d^T A d)d,\\
|d|=1,
\end{array}
\right.
\end{equation}
where $\varepsilon\in(0,1)$ is a positive parameter.

\begin{lemma}[\textbf{Local existence}]\label{ELLOC}
Let $\underline\rho$ and $\bar\rho$ be two positive numbers, and $d_*$ a constant unit vector.
Suppose that the initial data $(\rho_0, u_0, d_0)$ satisfies
\begin{eqnarray*}
  &&\underline\rho\leq\rho_0\leq\bar\rho,\quad \nabla\rho_0\in L^3(\mathbb R^3),\quad u_0\in D^{4,2}(\mathbb R^3),\\
  &&\text{div}\,u_0=0,\quad d_0-d_*\in D^{3,2}(\mathbb R^3),\quad |d_0|=1.
\end{eqnarray*}
Set
$$
g_0=\frac{1}{\sqrt{\rho_0}}[\Delta u_0-\nabla P_0-\text{div}(\nabla d_0\odot\nabla d_0+(\Delta d_0+|\nabla d_0|^2 d_0)\otimes d_0)],
$$
for some $P_0\in H^1(\mathbb R^3)$.

Then, there exists a local strong solution $(\rho, u,d)$ to system (\ref{EL}), subject to (\ref{bc})--(\ref{ic}), on $\mathbb R^3\times(0, T_\varepsilon)$, such that $\underline\rho\leq\rho\leq\bar\rho$ and
\begin{eqnarray*}
  &&(\nabla\rho, \partial_t\rho)\in L^\infty(0,T_\varepsilon; L^3(\mathbb R^3)),\quad u\in L^\infty(0,T_\varepsilon; D^{4,2}(\mathbb R^3)),\\
  &&\partial_tu\in L^2(0,T_\varepsilon; D^{2,2}(\mathbb R^3)),\quad\partial_td\in
  L^\infty(0,T_\varepsilon; H^1(\mathbb R^3))\cap L^2(0,T_\varepsilon; H^2(\mathbb R^3)),\\
  &&d-d_*\in L^\infty(0,T_\varepsilon; D^{3,2}(\mathbb R^3))\cap L^2(0,T_\varepsilon; D^{4,2}(\mathbb R^3)),
\end{eqnarray*}
and system (\ref{EL}) is satisfied pointwisely, a.e.~in $\mathbb R^3\times(0, T_{\varepsilon})$, where $T_\varepsilon$ is a positive
constant depending only $\varepsilon, \underline\rho, \bar\rho, \|\nabla u_0\|_{H^3}, \|\nabla d_0\|_{H^2}$ and $\|g_0\|_{L^2}$.
\end{lemma}

\begin{proof}
Note that, thanks to the presence of the biharmonic term $\varepsilon\Delta^2u$ in the momentum equations of system (\ref{EL}), the coupling terms between the velocity and the director in the momentum equations
are lower order terms, compared with the leading terms. Therefore,
one can use the standard linearization argument to show the local
well-posedness of strong solutions to system (\ref{EL}), subject to (\ref{bc})--(\ref{ic}). The proof is standard, and thus is omitted here.
\end{proof}

\begin{lemma}[\textbf{Basic energy identity}]\label{LEM1}
Let $(\rho, u, d)$ be a strong solution to system (\ref{EL}), on $\mathbb{R}^3\times (0, T)$. Then, it holds that
\begin{eqnarray}\label{BE}
\frac 12\frac{d}{dt}\int_{\mathbb{R}^3}(\rho |u|^2+|\n d|^2)dx+\int_{\mathbb{R}^3}(|\n u|^2+\varepsilon|\Delta u|^2+|\Delta d+|\n d|^2d|^2)dx=0,
\end{eqnarray}
for any $t\in(0, T)$.
\end{lemma}

\begin{proof}
Multiplying $(\ref{EL})_2$ by $u$, and integrating the resultant over $\mathbb{R}^3$, then it follows from integration by parts and using the
divergence-free condition that
\begin{eqnarray}
&&\frac12\frac{d}{dt}\int_{\mathbb{R}^3}\rho|u|^2dx+\int_{\mathbb{R}^3}(|\n u|^2+\varepsilon|\Delta u|^2)dx\nonumber\\
&=&\int_{\mathbb{R}^3}[\n d\odot \n d+(\Delta d+|\n d|^2d)\otimes d]:\n u dx.\label{li1}
\end{eqnarray}
Using the condition $|d|=1$, one can easily verify that $d\cdot\Delta d=-|\nabla d|^2$. It is straightforward to check the following identity
$$
\text{div}\,(\nabla d\odot\nabla d)=\nabla\left(\frac{|\nabla d|^2}{2}\right)+\nabla d\Delta d.
$$
Thanks to these, multiplying $(\ref{EL})_4$ by $-(\Delta d+|\n d|^2d)$, integrating the resultant over $\mathbb{R}^3$, then it follows from
integration by parts, and using the divergence-free condition and the constraint $|d|=1$ that
\begin{eqnarray*}
&&\frac12\frac{d}{dt}\int_{\mathbb{R}^3}|\n d|^2dx+\int_{\mathbb{R}^3}(|\Delta d+|\n d|^2d|^2)dx\\
&=&\int_{\mathbb R^3}[(u\cdot\nabla)d-(d\cdot\nabla)u+(d^TAd)d]\cdot(\Delta d+|\nabla d|^2d)dx\\
&=&\int_{\mathbb R^3}[(u\cdot\nabla)d\cdot\Delta d-(d\cdot\nabla)u\cdot\Delta d+(d^TAd)d\cdot\Delta d]dx\\
&=&\int_{\mathbb R^3}\bigg\{u\cdot\left[\text{div}\,(\nabla d\odot\nabla d)-\nabla\left(\frac{|\nabla d|^2}{2}\right)\right]-(d\cdot\nabla)u\cdot\Delta d -(d^TAd)|\nabla d|^2\bigg\}dx \\
&=&-\int_{\mathbb R^3}[\nabla u:\nabla d\odot\nabla d+(d\cdot\nabla)u\cdot(\Delta d+|\nabla d|^2 d)]dx\\
&=&-\int_{\mathbb{R}^3}[\n d\odot \n d+(\Delta d+|\n d|^2d)\otimes d]:\n u dx.
\end{eqnarray*}
Summing the above equality with (\ref{li1}) leads to the conclusion.
\end{proof}

\begin{lemma} \label{LEM2}
Let $(\rho, u, d)$ be a strong solution to system (\ref{EL}), on $\mathbb{R}^3\times (0, T)$. Then, for any $t\in(0,T)$, we have
\begin{eqnarray*}
&& \frac{d}{dt}\int_{\mathbb{R}^3}(|\n u|^2+\varepsilon|\Delta u|^2)dx+\int_{\mathbb{R}^3}\rho|\partial_tu|^2dx\\
&\leq&\eta(\|\nabla\partial_tu\|_{L^2}^2  +\|\nabla^2u\|_2^2) +C_\eta(\|\nabla u\|_2^6+\|\Delta d\|_2^2),
\end{eqnarray*}
for any $\eta\in(0,\frac12)$, where $C_\eta$ is a positive constant depending only on $\bar\rho$ and $\eta$, and
\begin{eqnarray*}
  && \frac{d}{dt}\int_{\mathbb R^3}|\Delta d|^2dx+\int_{\mathbb R^3}|\nabla\Delta d|^2dx\leq C(\|\nabla u\|_{L^2}^6+\|\nabla^2d\|_{L^2}^6+\|\nabla^2u\|_{L^2}^2),
\end{eqnarray*}
for an absolute positive constant $C$.
\end{lemma}

\begin{proof}
Multiplying $(\ref{EL})_2$ by $\partial_tu$, and integrating the resultant over $\mathbb R^3$, then it follows from integration by parts that
\begin{eqnarray}\label{ute}
&&\frac12\frac{d}{dt}\int_{\mathbb{R}^3}(|\n u|^2+\varepsilon|\Delta u|^2)dx+\int_{\mathbb{R}^3}\rho|\partial_tu|^2dx\nonumber\\
&=&\int_{\mathbb{R}^3}\big\{-\rho (u\cdot\n) u\cdot \partial_tu+[\n d\odot\n d+(\Delta d+|\n d|^2d)\otimes d]:\n \partial_tu\big\}dx\nonumber\\
&\leq&\int_{\mathbb R^3}[|\rho||u||\nabla u||\partial_tu|+(2|\nabla d|^2+|\Delta d|)|\nabla\partial_tu|]dx\nonumber\\
&\leq&\int_{\mathbb R^3}(|\rho||u||\nabla u||\partial_tu|+3 |\Delta d| |\nabla\partial_tu|)dx,
\end{eqnarray}
where in the last step, we have used the fact that $|\nabla d|^2=-d\cdot\Delta d\leq|\Delta d|$, guaranteed by the constraint $|d|=1$. By the H\"older, Sobolev and Young inequalities, we deduce
\begin{eqnarray*}
&&\int_{\mathbb R^3}(|\rho||u||\nabla u||\partial_tu|+3 |\Delta d| |\nabla\partial_tu|)dx\\
&\leq&\|\sqrt\rho\|_{L^\infty}\|u\|_{L^6}\|\nabla u\|_{L^3}\|\sqrt\rho\partial_t u\|_{L^2}+3\|\Delta d\|_{L^2}\|\nabla\partial_tu\|_{L^2}\\
&\leq&C\|\nabla u\|_{L^2}^{\frac32}\|\nabla^2 u\|_{L^2}^{\frac12}\|\sqrt\rho\partial_tu\|_{L^2}+3\|\Delta d\|_{L^2}\|\nabla\partial_tu\|_{L^2}\\
&\leq&\eta(\|\nabla\partial_tu\|_{L^2}^2+\|\sqrt\rho\partial_tu\|_{L^2}^2 +\|\nabla^2u\|_2^2) +C_\eta(\|\nabla u\|_2^6+\|\Delta d\|_2^2),
\end{eqnarray*}
for any $\eta>0$, where $C_\eta$ is a positive constant depending only on $\eta$ and $\bar\rho$. Substituting the above inequality into (\ref{ute}), and choosing $\eta\in(0,\frac12)$ yields the first conclusion.

Recalling that $|\nabla d|^2\leq|\Delta d|$, guaranteed by $|d|=1$, by simple calculations, one has
\begin{align*}
  |\nabla[(u\cdot\nabla)d&-(d\cdot\nabla)u+(d^TAd)d-|\nabla d|^2d]|\\
  \leq&C[(|u|+|\nabla d|)(|\nabla u|+|\nabla^2d|)+|\nabla^2u|].
\end{align*}
Applying the operator $\nabla$ to $(\ref{EL})_4$, multiplying the resultant by $-\nabla\Delta d$, and integrating over $\mathbb R^3$, then it follows from integration by parts, and using the above inequality that
\begin{eqnarray*}
  &&\frac12\frac{d}{dt}\int_{\mathbb R^3}|\Delta d|^2dx+\int_{\mathbb R^3}|\nabla\Delta d|^2dx\\
  &=&\int_{\mathbb R^3}\nabla[(u\cdot\nabla)d-(d\cdot\nabla)u+(d^TAd)d-|\nabla d|^2d] :\nabla\Delta d\,dx\\
  &\leq&C\int_{\mathbb R^3}[(|u|+|\nabla d|)(|\nabla u|+|\nabla^2d|)+|\nabla^2u|]|\nabla\Delta d|dx.
\end{eqnarray*}
By the H\"older, Sobolev and Young inequalities, we further deduce
\begin{eqnarray*}
  &&\frac12\frac{d}{dt}\int_{\mathbb R^3}|\Delta d|^2dx+\int_{\mathbb R^3}|\nabla\Delta d|^2dx\\
  &\leq&C[(\|u\|_{L^6}+\|\nabla d\|_{L^6})(\|\nabla u\|_{L^3}+\|\nabla^2d\|_{L^3})+\|\nabla^2u\|_{L^2}]\|\nabla\Delta d\|_{L^2}\\
  &\leq&C(\|\nabla u\|_{L^2}+\|\nabla^2d\|_{L^2})^{\frac32} (\|\nabla^2 u\|_{L^2}+\|\nabla\Delta d\|_{L^2})^{\frac12}\|\nabla\Delta d\|_{L^2}\\
  &&+C\|\nabla^2u\|_{L^2}]\|\nabla\Delta d\|_{L^2}\\
  &\leq&\frac12\|\nabla\Delta d\|_{L^2}^2+C(\|\nabla u\|_{L^2}^6+\|\nabla^2d\|_{L^2}^6+\|\nabla^2u\|_{L^2}^2),
\end{eqnarray*}
which implies the second conclusion. This completes the proof of Lemma \ref{LEM2}.
\end{proof}

\begin{lemma} \label{LEM3}
Let $(u, d, \rho)$ be a strong solutions to system $(\ref{EL})$, on
$\mathbb{R}^3\times(0,T)$. Then, we have the estimate
\begin{eqnarray*}
&&\|\n^2 u\|^2_{L^2}+\varepsilon\|\n^3u\|^2_{L^2}+\|\n^3 d\|^2_{L^2}\\
&\leq& C(\|\n \partial_td\|^2_{L^2}+\|\sqrt\rho \partial_tu\|^2_{L^2}+\|\n u\|^6_{L^2}+\|\nabla^2 d\|^6_{L^2}),
\end{eqnarray*}
for $t\in(0,T)$, where $C$ is a positive constant depending only on $\bar\rho$.
\end{lemma}

\begin{proof}
Multiplying$(\ref{EL})_2$ by $\Delta u$, and integrating the resultant over $\mathbb{R}^3$, then it follows from integration by parts that
\begin{eqnarray*}
\int_{\mathbb{R}^3}(|\n^2 u|^2+\varepsilon|\n^3u|^2)dx&=&\int_{\mathbb{R}^3}
\text{div}\,[\n d\odot \n d+(\Delta d+|\n d|^2d)\otimes d]\cdot \Delta u dx \nonumber\\
&&+\int_{\mathbb{R}^3}\rho(\partial_tu+(u\cdot\n) u) \cdot \Delta udx.
\end{eqnarray*}
Applying the operator $\n$ to $(\ref{EL})_4$, multiplying the resultant by $\n\Delta d$, and integrating over $\mathbb R^3$, then it follows from integration by parts that
\begin{align*}\label{n3de}
\int_{\mathbb{R}^3}|\n^3 d|^2dx
 = &\int_{\mathbb{R}^3}\n[\partial_td+(u\cdot\n) d-(d\cdot\n) u\nonumber\\
 &-|\n d|^2d+(d^T A d)d]:\n\Delta d\,dx.
\end{align*}
Summing the previous two equalities up yields
\begin{eqnarray}
  &&\int_{\mathbb R^3}(|\nabla^2u|^2+\varepsilon|\nabla^3u|^2+|\nabla^3d|^2)dx\nonumber\\
  &=&\int_{\mathbb R^3}  [\text{div}\,(\nabla d\odot\nabla d)+\rho(\partial_t u+(u\cdot\nabla)u)]\cdot\Delta udx\nonumber\\
  &&+ \int_{\mathbb R^3}\nabla[\partial_t d+(u\cdot\nabla)d -|\nabla d|^2d ]:\nabla\Delta d\,dx\nonumber\\
  &&+\int_{\mathbb R^3} \text{div}\,[(\Delta d+|\nabla d|^2d)\otimes d]\cdot\Delta udx\nonumber\\
  &&+\int_{\mathbb R^3}\nabla[(d^TAd)d-(d\cdot\nabla)u]:\nabla\Delta d\, dx\nonumber\\
  &=:&I_1+I_2+I_3+I_4.\label{li2}
\end{eqnarray}

Noticing that $\rho\leq\bar\rho$ and $|\nabla d|^2\leq|\Delta d|$, we have
\begin{eqnarray*}
  I_1+I_2&\leq&\int_{\mathbb R^3}[2(|\nabla d||\nabla^2d|+\rho|\partial_t u|+\rho|u||\nabla u|)|\Delta u|+(|\nabla\partial_td|\\
  &&+|u||\nabla^2d|+|\nabla u||\nabla d|+2|\nabla d||\nabla^2d|+|\nabla d|^3)|\nabla\Delta d|]dx\\
  &\leq&C(\bar\rho)\int_{\mathbb R^3}(|u|+|\nabla d|)(|\nabla u|+|\nabla^2d|)(|\Delta u|+|\nabla\Delta d|)dx\\
  &&+C(\bar\rho)\int_{\mathbb R^3}(\sqrt\rho|\partial_tu|+|\nabla\partial_t d|)(|\Delta u|+|\nabla\Delta d|)dx.
  \end{eqnarray*}
Recalling $|d|=1$, one can easily verify that
\begin{eqnarray}
  &&\Delta d+|\nabla d|^2 d=\Delta d-(d\cdot\Delta d)d=(d\times\Delta d)\times d,\label{NL1}\\
  &&(d\cdot\nabla)u-(d^TAd)d=(d\cdot\nabla)u-((d\cdot\nabla)u\cdot d)d  =(d\times(d\cdot\nabla)u)\times d.\label{NL2}
\end{eqnarray}
Thanks to these two equalities, it follows from integration by parts that
\begin{align*}
  I_3+ I_4
  =&\int_{\mathbb R^3} \text{div}\,\{[(d\times\Delta d)\times d]\otimes d\}\cdot\Delta udx\\
  &-\int_{\mathbb R^3}\nabla[(d\times(d\cdot\nabla)u)\times d]:\nabla\Delta d\, dx\\
  =&\int_{\mathbb R^3}\partial_k\{[(d\times\Delta d)\times d]\otimes d\}:\nabla\partial_k udx \\
  &-\int_{\mathbb R^3}\partial_k[(d\times(d\cdot\nabla)u)\times d]\cdot \partial_k\Delta d\,dx\\
  =&\int_{\mathbb R^3} [(d\times\Delta\partial_kd)\times d]\otimes d:\nabla\partial_kudx\\
  &-\int_{\mathbb R^3} [(d\times(d\cdot\nabla)\partial_ku)\times d]\cdot \partial_k\Delta d \,dx+I_r,
  \end{align*}
  where $I_r$ is the remaining term, which can be estimated as
  $$
  I_r\leq 3\int_{\mathbb R^3}( |\nabla d||\nabla^2d||\nabla^2u|+ |\nabla d||\nabla u||\nabla\Delta d|)dx.
  $$

Making use of the identity $(a\times b)\cdot c=(b\times c)\cdot a$ twice yields
$$
(d\cdot\nabla)\partial_ku\cdot [(d\times\Delta\partial_kd)\times d]=[(d\times(d\cdot\nabla)\partial_ku)\times d]\cdot \partial_k\Delta d,
$$
and thus we can further deduce
  \begin{align*}
  I_3+I_4
  =&\int_{\mathbb R^3}\{(d\cdot\nabla)\partial_ku\cdot [(d\times\Delta\partial_kd)\times d]\\
  &-[(d\times(d\cdot\nabla)\partial_ku)\times d]\cdot \partial_k\Delta d\}dx+I_r
  = I_r.
\end{align*}

Substituting the estimates for $I_1+I_2$ and $I_3+I_4$ into (\ref{li2}), it follows from the H\"older, Sobolev and Young inequalities that
\begin{eqnarray*}
&&\int_{\mathbb R^3}(|\nabla^2u|^2+\varepsilon|\nabla^3u|^2+|\nabla^3d|^2)dx
\leq I_1+I_2+I_r\nonumber\\
&\leq&C(\bar\rho)\int_{\mathbb R^3}(|u|+|\nabla d|)(|\nabla u|+|\nabla^2d|)(|\Delta u|+|\nabla\Delta d|)dx\\
  &&+C(\bar\rho)\int_{\mathbb R^3}(\sqrt\rho|\partial_tu|+|\nabla\partial_t d|)(|\Delta u|+|\nabla\Delta d|)dx\\
&\leq&C(\bar\rho)(\|u\|_{L^6}+\|\nabla d\|_{L^6})(\|\nabla u\|_{L^3}+ \|\nabla^2d\|_{L^3})(\|\Delta u\|_{L^2}+\|\nabla\Delta d\|_{L^2})\\
  &&+C(\bar\rho)(\|\sqrt\rho\partial_tu\|_{L^2}+\|\nabla\partial_td\|_{L^2}) (\|\Delta u\|_{L^2}+\|\nabla\Delta d\|_{L^2})\\
  &\leq&C(\bar\rho)(\|\nabla u\|_{L^2}+\|\nabla^2 d\|_{L^2})^{\frac32} (\|\nabla^2 u\|_{L^2}+\|\nabla^3 d\|_{L^2})^{\frac32}\\
  &&+C(\bar\rho)(\|\sqrt\rho\partial_tu\|_{L^2}+\|\nabla\partial_td\|_{L^2}) (\|\Delta u\|_{L^2}+\|\nabla\Delta d\|_{L^2})\\
  &\leq&\frac12(\|\nabla^2 u\|_{L^2}^2+\|\nabla^3 d\|_{L^2}^2)+C(\bar\rho)(\|\nabla u\|_{L^2}^6 +\|\nabla^2 d\|_{L^2}^6)\\
  && +C(\bar\rho)(\|\sqrt\rho\partial_tu\|_{L^2}^2+\|\nabla\partial_td\|_{L^2}^2),
\end{eqnarray*}
which implies the conclusion. This completes the proof of Lemma \ref{LEM3}.
\end{proof}

\begin{lemma}[\textbf{Estimate  for time derivatives}]\label{LEM4}
Let $(\rho, u, d)$ be a strong solution to system $(\ref{EL})$, on
$\mathbb{R}^3\times(0,T)$, then it holds that
\begin{align*}
\frac{d}{dt}\int_{\mathbb{R}^3} &(\rho |\partial_tu|^2+|\n \partial_td|^2)dx+\int_{\mathbb{R}^3}(|\n \partial_tu|^2+\varepsilon|\Delta \partial_tu|^2+|\Delta \partial_td|^2)dx\nonumber\\
\leq& C(\|\n u\|^4_{L^2}+\|\n^2 d\|^4_{L^2})(\|\sqrt\rho\partial_tu\|^2_{L^2}+\|\n^2u\|_{L^2}^2+\|\n \partial_td\|^2_{L^2} ),
\end{align*}
for $t\in(0,T)$, where $C$ is a positive constant depending only on $\bar\rho$.
\end{lemma}

\begin{proof}
Differentiating$(\ref{EL})_2$ with respect to $t$, we obtain
\begin{align*}
\rho(\partial_t^2u&+(u\cdot\n) \partial_tu)+\rho (\partial_tu\cdot\nabla) u+\partial_t\rho (\partial_tu+ (u\cdot\n) u)+\n P_t\\
=&\Delta \partial_tu-\varepsilon\Delta^2\partial_tu-\text{div}\,\partial_t[\n d\odot \n d+(\Delta d+|\n d|^2d)\otimes d].
\end{align*}
Multiplying the above equation by $\partial_tu$, integrating the resultant over $\mathbb R^3$, and using the continuity equation, i.e., $(\ref{EL})_1$, it follows from integration by parts that
\begin{eqnarray*}
&&\frac12\frac{d}{dt}\int_{\mathbb{R}^3}\rho |\partial_tu|^2dx+\int_{\mathbb{R}^3}(|\n \partial_tu|^2+\varepsilon|\Delta \partial_tu|^2)dx \nonumber\\
&=&\int_{\mathbb{R}^3}\big\{\mbox{div}(\rho u)(\partial_tu+(u\cdot \n) u)\cdot \partial_tu-\rho (\partial_tu\cdot\n) u\cdot \partial_tu\nonumber\\
&&\ \ \ \ +\partial_t [\n d\odot \n d+(\Delta d+|\n d|^2d)\otimes d] :\n \partial_tu\big\}dx.
\end{eqnarray*}
Differentiating $(\ref{EL})_4$ with respect to t, multiplying the resultant by $\Delta \partial_td$, and integrating over $\mathbb{R}^3$,
then it follows from integration by parts that
\begin{eqnarray*}
&&\frac12\frac{d}{dt}\int_{\mathbb{R}^3}|\n \partial_td|^2dx+\int_{\mathbb{R}^3}|\Delta \partial_td|^2dx\nonumber\\
&=&\int_{\mathbb{R}^3}\partial_t[(u\cdot\n) d-(d\cdot\n) u-|\n d|^2d+(d^T A d)d]\cdot \Delta \partial_td\,dx.
\end{eqnarray*}
Summing the above two equalities up it yields
\begin{eqnarray}\label{utdt}
&&\frac12\frac{d}{dt}\int_{\mathbb{R}^3} (\rho |\partial_tu|^2+|\n \partial_td|^2)dx+\int_{\mathbb{R}^3}(|\n \partial_tu|^2+\varepsilon|\Delta \partial_tu|^2+|\Delta \partial_td|^2dx\nonumber\\
&=&\int_{\mathbb{R}^3}\mbox{div}(\rho u)(\partial_tu+(u\cdot \n) u)\cdot \partial_tudx-\int_{\mathbb{R}^3}\rho (\partial_tu\cdot\n) u\cdot \partial_tudx\nonumber\\
&\ &+\int_{\mathbb{R}^3} (\n \partial_td\odot \n d+ \n d\odot\n \partial_td): \n \partial_tudx+\int_{\mathbb{R}^3}\Delta \partial_td\otimes d:\n \partial_tudx\nonumber\\
&\ &+2\int_{\mathbb{R}^3} (\n \partial_td :\n d)d\otimes d:\n \partial_tudx +\int_{\mathbb{R}^3} |\n d|^2\partial_td\otimes d:\n \partial_tudx\nonumber\\
&\ &+\int_{\mathbb{R}^3} \Delta d\otimes \partial_td:\n \partial_tudx +\int_{\mathbb{R}^3} |\n d|^2d\otimes \partial_td : \n \partial_tudx\nonumber\\
&\ &+\int_{\mathbb{R}^3} (\partial_tu\cdot \n) d\cdot \Delta \partial_tddx+\int_{\mathbb{R}^3} (u\cdot \n) \partial_td\cdot \Delta \partial_tddx\nonumber\\
&\ &-\int_{\mathbb{R}^3} (\partial_td\cdot \n) u\cdot \Delta \partial_tddx -\int_{\mathbb{R}^3} (d\cdot \n) \partial_tu \cdot \Delta \partial_tddx\nonumber\\
&\ &-\int_{\mathbb{R}^3}2(\n \partial_td : \n d)d\cdot \Delta \partial_tddx-\int_{\mathbb{R}^3} |\n d|^2\partial_td\cdot \Delta \partial_tddx\nonumber\\
&\ &+\int_{\mathbb{R}^3} ((d^T A d)d)_t\cdot \Delta \partial_tddx=:\sum^{15}_{i=1}J_i.
\end{eqnarray}

We are going to estimate the terms $J_i, i=1,2,\cdots, 15$. For $J_1$, it follows from integration by parts, and using the H\"older, Sobolev and Young inequalities that
\begin{align*}
J_1=&\int_{\mathbb{R}^3} \mbox{div}(\rho u)(\partial_tu+(u\cdot \n) u)\cdot \partial_tudx \\
=& -\int_{\mathbb{R}^3} \rho u\cdot\n[(\partial_tu+(u\cdot\n) u)\cdot \partial_tu]dx\\
\leq&\int_{\mathbb{R}^3} (2\rho|u||\n \partial_tu||\partial_tu|+\rho|u||\n u|^2|\partial_tu|\\
&+\rho|u|^2|\n^2 u||\partial_tu|+\rho|u|^2|\n u||\n \partial_tu|)dx\\
\leq&2\|\sqrt{\rho}u\|_{L^6}\|\sqrt{\rho}\partial_tu\|_{L^2}^{\frac12} \|\sqrt{\rho}\partial_tu\|_{L^6}^{\frac12} \|\n \partial_tu\|_{L^2}+\|\rho\|_{L^\infty}\|u\|_{L^6}\|\nabla u\|_{L^3}^2\|\partial_tu\|_{L^6} \\
&+\|\rho\|_{L^\infty}\|u\|_{L^6}^2\|\n^2 u\|_{L^2} \|\partial_tu\|_{L^6}+\|\rho\|_{L^\infty}\|u\|_{L^6}^2\|\n u\|_{L^6} \|\n \partial_tu|_{L^2} \\
\leq&C(\|\nabla u\|_{L^2}\|\sqrt\rho\partial_tu\|_{L^2}^{\frac12}\|\nabla\partial_tu \|_{L^2}^{\frac32}+\|\nabla u\|_{L^2}^2\|\nabla^2u\|_{L^2}\|\nabla\partial_tu\|_{L^2})\\
\leq&\eta\|\n \partial_tu\|^2_{L^2}+C_\eta\|\n u\|^4_{L^2}(\|\n^2 u\|^2_{L^2}+\|\sqrt\rho \partial_tu\|^2_{L^2}),
\end{align*}
for any positive $\eta\in(0,1)$, where $C_\eta$ is a positive constant depending only on $\eta$ and $\bar\rho$. For $J_2$ and $J_3+J_5$, by the H\"older, Sobolev and Young inequalities, we deduce
\begin{eqnarray*}
J_2&=&-\int_{\mathbb{R}^3} \rho (\partial_tu\cdot\n) u\cdot \partial_tudx \leq \|\n u\|_{L^2}\|\rho \partial_tu\|^{\frac12}_{L^2}\|\rho \partial_tu\|^{\frac12}_{L^6}\|\partial_tu\|_{L^6}\\
&\leq&C\|\n u\|_{L^2}\|\rho \partial_tu\|^{\frac12}_{L^2}\|\n \partial_tu\|^{\frac32}_{L^2} \leq \eta\|\n \partial_tu\|^2_{L^2}+C_\eta\|\n u\|^4_{L^2}\|\sqrt\rho \partial_tu\|^2_{L^2},
\end{eqnarray*}
and
\begin{eqnarray*}
J_3+J_5&=&\int_{\mathbb{R}^3} (\n \partial_td\odot\n d+\n d\odot \n \partial_td):\n \partial_tu \\
&&+2 \n \partial_td:\n d)d\otimes d:\n \partial_tudx \leq 4\int_{\mathbb{R}^3} |\n \partial_td||\n d||\n \partial_tu|dx\\
&\leq&4\|\nabla\partial_td\|_{L^3}\|\nabla d\|_{L^6}\|\nabla\partial_tu\|_{L^2}\\
&\leq& C\|\nabla\partial_td\|_{L^2}^{\frac12}\|\Delta\partial_td\|_{L^2}^{\frac12} \|\nabla^2d\|_{L^2}\|\nabla\partial_tu\|_{L^2}\\
&\leq&\eta(\|\n \partial_tu\|^2_{L^2}+\|\Delta\partial_t d\|_{L^2}^2)+C_\eta\|\nabla^2d\|_{L^2}^4\|\n \partial_td\|^2_{L^2},
\end{eqnarray*}
for any positive $\eta\in(0,1)$, where $C_\eta$ is a positive constant depending only on $\eta$ and $\bar\rho$.
For $J_4+J_{12}$, integration by parts yields
\begin{eqnarray}
J_4+J_{12}&=&\int_{\mathbb{R}^3} [(\Delta \partial_td\otimes d):\n \partial_tu-(d\cdot \n)\partial_tu\cdot\Delta \partial_td]dx\nonumber\\
&=&\int_{\mathbb{R}^3}  (\Delta \partial_t d^i d^j \partial_j\partial_tu^i - d^j \partial_j\partial_tu^i\Delta\partial_t d^i)dx = 0.\nonumber
\end{eqnarray}
By the H\"older, Sobolev and Young inequalities, we can estimate $J_6+J_8+J_{14}$ as follows
\begin{eqnarray*}
J_6+J_8+J_{14}&=&\int_{\mathbb{R}^3}|\n d|^2 [(\partial_td\otimes d+d\otimes\partial_td):\n \partial_tu-\partial_td\cdot\Delta\partial_td]dx\\
&\leq&2\|\n d\|^2_{L^6}\|\partial_td\|_{L^6}(\|\n \partial_tu\|_{L^2}+\|\Delta\partial_td\|_{L^2})\\
 &\leq& C \|\n^2d\|^2_{L^2}\|\n \partial_td\|_{L^2}(\|\n \partial_tu\|_{L^2}+\|\Delta\partial_td\|_{L^2})\\
&\leq&\eta(\|\n \partial_tu\|^2_{L^2}+\|\Delta\partial_td\|_{L^2}^2) +C_\eta\|\n^2d\|^4_{L^2}\|\n \partial_td\|^2_{L^2},
\end{eqnarray*}
for any positive $\eta\in(0,1)$, where $C_\eta$ is a positive constant depending only on $\eta$. Using the Gagliardo-Nirenberg inequality, $\|f\|_{L^\infty(\mathbb R^3)}\leq C\|f\|_{L^6(\mathbb R^3)}^{\frac12}\|\Delta f\|_{L^2(\mathbb R^3)}^{\frac12}$, it follows from the H\"older, Sobolev and Young inequalities that
\begin{eqnarray*}
J_7+J_{11}&=&\int_{\mathbb{R}^3} [\Delta d\otimes \partial_td:\n \partial_tu-(\partial_td\cdot\nabla)u\cdot\Delta\partial_td]dx\\
&\leq &(\|\Delta d\|_{L^2}+\|\nabla u\|_{L^2})\|\partial_td\|_{L^\infty}(\|\n \partial_tu\|_{L^2} +\|\Delta\partial_td\|_{L^2} )\\
&\leq&C(\|\Delta d\|_{L^2}+\|\nabla u\|_{L^2})\|\partial_td\|_{L^6}^{\frac12}\|\Delta\partial_t d\|_{L^2}^{\frac12}(\|\n \partial_tu\|_{L^2} +\|\Delta\partial_td\|_{L^2} ) \\
&\leq& C(\|\Delta d\|_{L^2}+\|\nabla u\|_{L^2})\|\nabla\partial_td\|_{L^2}^{\frac12}\|\Delta\partial_t d\|_{L^2}^{\frac12}(\|\n \partial_tu\|_{L^2} +\|\Delta\partial_td\|_{L^2}^2)\\
&\leq&\eta(\|\n \partial_tu\|^2_{L^2}+\|\Delta\partial_td\|_{L^2}^2) +C_\eta(\|\Delta d\|_{L^2}^4+\|\nabla u\|_{L^2}^4)\|\nabla\partial_td\|^2_{L^2},
\end{eqnarray*}
for any positive $\eta\in(0,1)$, where $C_\eta$ is a positive constant depending only on $\eta$. For $J_9$, we integrate by parts, and using the H\"older, Sobolev and Young inequality to deduce
\begin{eqnarray*}
J_9&=&\int_{\mathbb{R}^3}  (\partial_tu\cdot \n) d\cdot\Delta \partial_td\,dx
 = -\int_{\mathbb{R}^3}\nabla[(\partial_tu\cdot \n) d]:\nabla  \partial_td\,dx\\
&\leq&\int_{\mathbb{R}^3} (|\n \partial_tu||\n d||\n \partial_td|+|\partial_tu||\n^2 d||\n \partial_td|)dx\\
&\leq&\|\n \partial_tu\|_{L^2}\|\n d\|_{L^6}\|\n \partial_td\|_{L^3}+\|\partial_tu\|_{L^6}\|\n^2 d\|_{L^2}\|\n \partial_td\|_{L^3}\\
&\leq&C\|\n \partial_tu\|_{L^2}\|\n^2 d\|_{L^2}\|\n \partial_td\|^{\frac12}_{L^2}\|\Delta\partial_td\|^{\frac12}_{L^2}\\
&\ &+C\|\nabla\partial_tu\|_{L^2}\|\n^2 d\|_{L^2}\|\n \partial_td\|_{L^2}^{\frac12}\|\n^2 \partial_td\|^{\frac12}_{L^2}\\
&\leq&\eta(\|\n \partial_tu\|^2_{L^2}+\|\Delta \partial_td\|^2_{L^2})+C_\eta\|\n^2 d\|^4_{L^2}\|\n \partial_td\|^2_{L^2},
\end{eqnarray*}
for any positive $\eta\in(0,1)$, where $C_\eta$ is a positive constant depending only on $\eta$. For $J_{10}+J_{13}$, by the H\"older, Sobolev and Young inequalities, we deduce
\begin{eqnarray*}
J_{10}+J_{13}&=&\int_{\mathbb{R}^3}  [u\cdot \n \partial_td\cdot \Delta \partial_td-2(\n \partial_td:\n d)d\cdot \Delta \partial_td]dx\\
&\leq&2(\|\nabla d\|_{L^6}+\|u\|_{L^6})\|\nabla\partial_td\|_{L^3} \|\Delta\partial_td\|_{L^2}\\
&\leq&C(\|\nabla^2 d\|_{L^2}+\|\nabla u\|_{L^2})\|\nabla\partial_td\|_{L^2}^{\frac12} \|\Delta\partial_td\|_{L^2}^{\frac32}\\
&\leq&\eta\|\Delta\partial_td\|_{L^2}^2+C_\eta(\|\nabla^2d\|_{L^2}^4 +\|\nabla u\|_{L^2}^4)\|\nabla\partial_td\|_{L^2}^2,
\end{eqnarray*}
for any positive $\eta\in(0,1)$, where $C_\eta$ is a positive constant depending only on $\eta$. Using the constraint $|d|=1$, we have
\begin{align*}
  d\cdot\Delta\partial_td=\partial_t(d\cdot\Delta d)-\partial_td\cdot\Delta d=-\partial_t(|\nabla d|^2)-\partial_td\cdot\Delta d.
\end{align*}
By the aid of this identity, one obtains
\begin{eqnarray*}
  J_{15}&=&\int_{\mathbb R^3}\partial_t[(d^TAd)d]\cdot\Delta\partial_td\,dx\\
  &=&\int_{\mathbb R^3}[(d^T\partial_tAd)d\cdot\Delta\partial_td +\partial_t(d^id^jd^l)\partial_ju^i \Delta\partial_td^l]dx\\
  &=&\int_{\mathbb R^3}[\partial_t(d^id^jd^l)\partial_ju^i \Delta\partial_td^l-(d^T\partial_tAd)(\partial_t(|\nabla d|^2)+\partial_t d\cdot \Delta d)]dx\\
  &\leq&3\int_{\mathbb R^3}[|\partial_td||\nabla u||\Delta\partial_td| + |\nabla\partial_tu|(|\nabla d||\nabla\partial_td|+|\partial_td||\Delta d|)]dx.
\end{eqnarray*}
Hence, by the H\"older, Sobolev and Young inequality, as well as the Gagliardo-Nirenber inequality, $\|f\|_{L^\infty(\mathbb R^3)}\leq C\|f\|_{L^6(\mathbb R^2)}^{\frac12}\|\Delta f\|_{L^2(\mathbb R^3)}^{\frac12}$, we can estimate $J_{15}$ as
\begin{eqnarray*}
  J_{15}&\leq&3\|\partial_td\|_{L^\infty}(\|\nabla u\|_{L^2}+\|\Delta d\|_{L^2})(\|\Delta\partial_td\|_{L^2}+\|\nabla\partial_tu\|_{L^2})\\
  &&+3\|\nabla\partial_tu\|_{L^2}\|\nabla d\|_{L^6}\|\nabla\partial_t d\|_{L^3}\\
  &\leq&C\|\partial_td\|_{L^6}^{\frac12}\|\Delta\partial_td \|_{L^2}^{\frac12}(\|\nabla u\|_{L^2}+\|\Delta d\|_{L^2})(\|\Delta\partial_td\|_{L^2}+\|\nabla\partial_tu\|_{L^2})\\ &&+C\|\nabla\partial_tu\|_{L^2}\|\nabla^2 d\|_{L^2}\|\nabla\partial_t d\|_{L^2}^{\frac12}\|\Delta\partial_t d\|_{L^2}^{\frac12}\\
  &\leq&C\|\nabla\partial_td\|_{L^2}^{\frac12}(\|\nabla u\|_{L^2}+\|\Delta d\|_{L^2})(\|\Delta\partial_td\|_{L^2} +\|\nabla\partial_tu\|_{L^2})^{\frac32}\\ &&+C\|\nabla\partial_tu\|_{L^2}\|\nabla^2 d\|_{L^2} \|\nabla\partial_t d\|_{L^2}^{\frac12}\|\Delta\partial_t d\|_{L^2}^{\frac12}\\
  &\leq&\eta(\|\Delta\partial_td\|_{L^2}^2  +\|\nabla\partial_tu\|_{L^2}^2)+C_\eta(\|\nabla u\|_{L^2}^4+\|\nabla^2d\|_{L^2}^4)\|\nabla\partial_td\|_{L^2}^2,
\end{eqnarray*}
for any positive $\eta\in(0,1)$, where $C_\eta$ is a positive constant depending only on $\eta$.

Thanks to the above estimates for $J_i, i=1,2,\cdots,15$, choosing $\eta$ small enough, it follows from (\ref{utdt}) that
\begin{align*}
\frac{d}{dt}\int_{\mathbb{R}^3} &(\rho |\partial_tu|^2+|\n \partial_td|^2)dx+\int_{\mathbb{R}^3}(|\n \partial_tu|^2+\varepsilon|\Delta \partial_tu|^2+|\Delta \partial_td|^2dx\nonumber\\
\leq& C(\|\n u\|^4_{L^2}+\|\n^2 d\|^4_{L^2})(\|\sqrt\rho\partial_tu\|^2_{L^2}+\|\n^2u\|_{L^2}^2+\|\n \partial_td\|^2_{L^2} ),
\end{align*}
for a positive constant $C$ depending only on $\bar\rho$. This completes the proof of Lemma \ref{LEM4}.
\end{proof}

\begin{proposition}\label{LEM5}
Under the same conditions in Lemma \ref{ELLOC}, the unique strong solution $(\rho, u, d)$ established there can be extended to another time $T_*$, which depends only on $\bar\rho$, $\|g_0\|_{L^2(\mathbb R^3)}$, $\|\nabla u_0\|_{H^1}, \|\nabla d_0\|_{H^2}$ and $\varepsilon\|\Delta^2u_0\|_{L^2}$. Moreover, we have the following estimate
\begin{align*}
  &\sup_{0\leq t\leq T_*}(\|u\|_{D^{2,2}}^2+\|d-d_*\|_{D^{3,2}}^2+\|\nabla\partial_td\|_{L^2}^2+\|\nabla\rho\|_{L^3}^3+\|\partial_t\rho\|_{L^3}^3) \\
  &+\int_0^{ T_*}(\|\nabla\partial_tu\|_{L^2}^2+\|\nabla\Delta u\|_{L^2}^2+\|\Delta\partial_td\|_{L^2}^2+\|\Delta^2d\|_{L^2}^2)ds\leq C,
\end{align*}
where $C$ is a positive constant depending only on $\bar\rho$, $\|g_0\|_{L^2}$, $\|\nabla u_0\|_{H^1}, \|\nabla d_0\|_{H^2}$ and $\varepsilon\left\|\frac{\Delta^2 u_0}{\sqrt{\rho_0}}\right\|_{L^2}$.
\end{proposition}

\begin{proof}
Extend the strong solution $(\rho, u, d)$ in Lemma \ref{ELLOC} to the maximal existence time $\tilde T$. Define a function $f_\varepsilon$ on $[0,\tilde T)$ as
\begin{align*}
  f_\varepsilon(t)=&(\|\nabla u\|_{L^2}^2+\varepsilon\|\Delta u\|_{L^2}^2+\| \sqrt\rho\partial_tu\|_{L^2}^2+\|\Delta d\|_{L^2}^2+ \|\nabla\partial_td\|_{L^2}^2)(t)\\
  &+\frac12\int_0^t(\|\sqrt\rho\partial_tu\|_{L^2}^2+\|\nabla\partial_tu\|_{L^2}^2+\varepsilon \| \Delta\partial_tu\|_{L^2}^2+\|\nabla\Delta d\|_{L^2}^2+\|\Delta\partial_td\|_{L^2}^2)ds.
\end{align*}
Then, by Lemma \ref{LEM2} and Lemma \ref{LEM4}, we obtain
$$
f_\varepsilon'(t)\leq C(f_\varepsilon^3(t)+f_\varepsilon(t)+\|\nabla^2u\|_{L^2}^2(t)) +Cf_\varepsilon^2(t) (f_\varepsilon(t)+\|\nabla^2u\|_{L^2}^2(t)),
$$
for any $t\in(0,\tilde T)$, and for a positive constant $C$ depending only on $\bar\rho$.
By Lemma \ref{LEM3}, one has $\|\nabla^2u\|_{L^2}^2(t)\leq C(f_\varepsilon(t)+f_\varepsilon^3(t))$, for any $t\in(0,T_\varepsilon)$, and for a positive constant $C$ depending only on $\bar\rho$.
Therefore, by the Young inequality, we have
\begin{eqnarray*}
(1+f_\varepsilon)'(t)&\leq& C(f_\varepsilon^3(t)+f_\varepsilon(t)) +Cf_\varepsilon^2(t)(f_\varepsilon(t)+f_\varepsilon^3(t))\nonumber\\
&\leq& C(1+f_\varepsilon^5(t))\leq C_*(1+f_\varepsilon(t))^5,
\end{eqnarray*}
for $t\in(0,\tilde T)$, where $C$ and $C_*$ are positive constants depending only on $\bar\rho$.
Set $h_\varepsilon(t)=f_\varepsilon(t)+1$. Then, the above ordinary differential inequality implies that
\begin{equation}\label{NNL1}
h_\varepsilon(t)\leq(h_\varepsilon^{-4}(0)-4C_*t)^{-\frac14}\leq 2^{\frac14}h_\varepsilon(0), \quad\forall t\in[0, t_*'),
\end{equation}
where $t_*'=\min\left\{\frac{1}{8C_*h_\varepsilon^4(0)},\tilde T\right\}$.

We are going to estimate $h_\varepsilon(0)$. Using equation $(\ref{EL})_4$, and by the H\"older and Sobolev inequalities, one can easily verify that $\|\nabla\partial_td(0)\|_{L^2}\leq C$, for a positive constant $C$ depending only on $\|\nabla u_0\|_{H^1}+\|\nabla d_0\|_{H^2}$. Hence, we have
$$
h_\varepsilon(0)\leq C+\|\sqrt\rho\partial_tu\|_{L^2}^2(0),
$$
for a positive constant $C$ depending only on $\|\nabla u_0\|_{H^1}+\|\nabla d_0\|_{H^2}$.
We still need estimate $\|\sqrt\rho\partial_tu\|_{L^2}^2(0)$. To this end, we
multiply equation $(\ref{EL})_2$ by $\partial_tu$, and integrate over $\mathbb R^3$, then it follows from integration by parts that
\begin{align*}
  &\|\sqrt\rho\partial_tu\|_{L^2}^2(t)\\
  =&\int_{\mathbb R^3}[\Delta u-\nabla P-\text{div}(\nabla d\odot\nabla d+(\Delta d+|\nabla d|^2d)\otimes d)-\varepsilon\Delta^2u-\rho(u\cdot\nabla)u]\cdot\partial_tu dx \\
  =&\int_{\mathbb R^3}[\Delta u-\nabla P_0-\text{div}(\nabla d\odot\nabla d+(\Delta d+|\nabla d|^2d)\otimes d)-\varepsilon\Delta^2u -\rho(u\cdot\nabla)u]\cdot\partial_tu dx \\
  \leq&\|\sqrt\rho\partial_tu\|_{L^2}\left\|\frac{1}{\sqrt\rho}[\Delta u-\nabla P_0-\text{div}(\nabla d\odot\nabla d+(\Delta d+|\nabla d|^2d)\otimes d)]\right\|_{L^2}\\
  &+\varepsilon\|\sqrt\rho\partial_tu\|_{L^2}\left\|\frac{\Delta^2u}{\sqrt\rho} \right\|_{L^2}+\|\sqrt\rho\partial_tu\|_{L^2}\|\sqrt\rho (u\cdot\nabla)u\|_{L^2}.
\end{align*}
Thus
\begin{align*}
  \|\sqrt\rho\partial_tu\|_{L^2}(t)\leq&\left\|\frac{1}{\sqrt\rho}[\Delta u-\nabla P_0-\text{div}(\nabla d\odot\nabla d+(\Delta d+|\nabla d|^2d)\otimes d)]\right\|_{L^2}\\
  &+\varepsilon \left\|\frac{\Delta^2u}{\sqrt\rho} \right\|_{L^2}+ \|\sqrt\rho (u\cdot\nabla)u\|_{L^2},
\end{align*}
from which, by taking $t\rightarrow0^+$, and using the H\"older and Sobolev inequalities, one obtains
$$
 \|\sqrt\rho\partial_tu\|_{L^2}(0)\leq \|g_0\|_{L^2}+\varepsilon \left\|\frac{\Delta^2u_0}{\sqrt{\rho_0}} \right\|_{L^2}+ C,
$$
for a positive constant $C$ depending only on $\|\nabla u_0\|_{H^1}$ and $\bar\rho$.

Thanks to the above estimate, it follows from (\ref{NNL1}) that
\begin{align*}
(\|\nabla u\|_{L^2}^2& +\| \sqrt\rho\partial_tu\|_{L^2}^2+\|\Delta d\|_{L^2}^2+ \|\nabla\partial_td\|_{L^2}^2)(t)\\
  &+ \int_0^t(\|\nabla\partial_tu\|_{L^2}^2 +\|\nabla\Delta d\|_{L^2}^2+\|\Delta\partial_td\|_{L^2}^2)ds\leq C,
\end{align*}
for any $t\in(0,\min\{t_*, \tilde T\})$, where $t_*$ and $C$ are positive constants depending only on $\bar\rho$, $\|\nabla u_0\|_{H^1}+\|\nabla d_0\|_{H^2}$, $\|g_0\|_{L^2}$ and $\varepsilon \left\|\frac{\Delta^2u_0}{\sqrt{\rho_0}} \right\|_{L^2}$. By Lemma \ref{LEM1} and Lemma \ref{LEM3}, it follows from the above estimate that
\begin{align}
(\|\nabla u\|_{H^1}^2&+\|\nabla d\|_{H^2}^2+ \|\nabla\partial_td\|_{L^2}^2)(t)  \nonumber\\
  &+\int_0^{t}(\|\nabla\partial_tu\|_{L^2}^2 +\|\nabla\Delta d\|_{L^2}^2+\|\Delta\partial_td\|_{L^2}^2)ds\leq C,\label{NNL2}
\end{align}
for any $t\in(0,\min\{t_*, \tilde T\})$, and $C$ is a positive constant depending only on $\bar\rho$, $\|\nabla u_0\|_{H^1}+\|\nabla d_0\|_{H^2}$, $\|g_0\|_{L^2}$ and $\varepsilon \left\|\frac{\Delta^2u_0}{\sqrt{\rho_0}} \right\|_{L^2}$.

We now work on the estimates of $\|\nabla\Delta u\|_{L^2}^2$ and $\|\Delta^2d\|_{L^2}^2$. Applying the operator $\n$ to $(\ref{EL})_1$, multiplying the resultant by $\n\Delta u$ and integrating over $\mathbb{R}^3$ yields
\begin{eqnarray*}
 \int_{\mathbb{R}^3} (|\n\Delta u|^2+\varepsilon|\Delta^2u|^2)dx
&=&\int_{\mathbb{R}^3}\nabla[\rho(\partial_tu+(u\cdot\n) u)+\text{div}(\nabla d\odot\nabla d)]:\n\Delta udx\\
&& +\int_{\mathbb{R}^3}\n\text{div}[(\Delta d+|\n d|^2)\otimes d]:\n\Delta udx.
\end{eqnarray*}
Applying the operator $\Delta$ to $(\ref{EL})_4$, multiplying the resultant by $\Delta^2 d$ and integrating over $\mathbb{R}^3$ yields
\begin{eqnarray*}
\int_{\mathbb{R}^3}|\Delta^2d|^2dx&=&\int_{\mathbb{R}^3}\Delta[\partial_td+(u\cdot\nabla)d-|\nabla d|^2 d]\cdot\Delta^2d dx \\
&&-\int_{\mathbb R^3}\Delta[(d\cdot\nabla)u-(d^TAd)d]\cdot\Delta^2 ddx.
\end{eqnarray*}
Summing the previous two equalities up, recalling that $|\nabla d|^2\leq|\Delta d|$, after some simple calculations, one obtains
\begin{align}
  &\int_{\mathbb R^3}(|\nabla\Delta u|^2+\varepsilon|\Delta^2u|^2+|\Delta^2d|^2)dx\nonumber \\
  \leq&\int_{\mathbb{R}^3}\{\n\text{div}[(\Delta d+|\n d|^2)\otimes d]:\n\Delta u-\Delta[(d\cdot\nabla)u-(d^TAd)d]\cdot\Delta^2 d\}dx
  \nonumber\\
  &+\int_{\mathbb R^3}|\nabla\rho|(|\partial_tu|+|u||\nabla u|)|\nabla\Delta u|dx +C\int_{\mathbb R^3}[|\nabla\partial_tu|+|\Delta\partial_td|+(|u|+|\nabla d|)\nonumber\\
  &\times(|\nabla^2u|+|\nabla^3d|)+|\nabla u|^2+|\nabla^2d|^2](|\nabla\Delta u|+|\Delta^2d|)dx,\label{NL3}
\end{align}
for a positive constant $C$ depending only $\bar\rho$.

Recalling the equalities (\ref{NL1})--(\ref{NL2}), one has
\begin{eqnarray*}
  &&\int_{\mathbb{R}^3}\big\{\n\text{div}[(\Delta d+|\n d|^2)\otimes d]:\n\Delta u- \Delta[(d\cdot\nabla)u-(d^TAd)d]\cdot\Delta^2 d \big\}dx\\
  &=&\int_{\mathbb R^3}\big\{\Delta[((d\times\Delta d)\times d)\otimes d]:\nabla\Delta u-\Delta[(d\times(d\cdot\nabla)u)\times d]\cdot\Delta^2 d\big\}dx\\
  &=&\int_{\mathbb R^3}\big\{[((d\times\Delta^2 d)\times d)\otimes d]:\nabla\Delta u-[(d\times(d\cdot\nabla)\Delta u)\times d]\cdot\Delta^2 d\big\}dx+K_r,
\end{eqnarray*}
where $K_r$ is an integral, which, by straightforward calculations, and using
$|\nabla d|^2\leq|\Delta d|$ and the Young inequality, can be bounded by
$$
  |K_r|\leq 100\int_{\mathbb R^3}[|\nabla d|(|\nabla^2u|+|\nabla^3d|)+|\nabla u|^2+|\nabla^2d|^2](|\nabla\Delta u|+|\Delta^2d|)dx.
$$
Applying the identity $(a\times b)\cdot c=(b\times c)\cdot a$ twice yields
$$
[((d\times\Delta^2 d)\times d)\otimes d]:\nabla\Delta u=[d\times((d\cdot\nabla)\Delta u)\times d]\cdot\Delta^2 d,
$$
and thus we have the following estimate
\begin{eqnarray*}
  &&\int_{\mathbb{R}^3}\big\{\n\text{div}[(\Delta d+|\n d|^2)\otimes d]:\n\Delta u- \Delta[(d\cdot\nabla)u-(d^TAd)d]\cdot\Delta^2 d \big\}dx\\
  &=&K_r\leq 100\int_{\mathbb R^3}[|\nabla d|(|\nabla^2u|+|\nabla^3d|)+|\nabla u|^2+|\nabla^2d|^2](|\nabla\Delta u|+|\Delta^2d|)dx.
\end{eqnarray*}

Thanks to the above estimate, it follows from (\ref{NL3}) that
\begin{eqnarray*}
  &&\int_{\mathbb R^3}(|\nabla\Delta u|^2+\varepsilon|\Delta^2u|^2+|\Delta^2d|^2)dx\nonumber \\
  &\leq&\int_{\mathbb R^3}|\nabla\rho|(|\partial_tu|+|u||\nabla u|)|\nabla\Delta u|dx+C\int_{\mathbb R^3}[|\nabla\partial_tu|+|\Delta\partial_td|+(|u|+|\nabla d|) \nonumber\\
  &&\times(|\nabla^2u|+|\nabla^3d|) +|\nabla u|^2+|\nabla^2d|^2](|\nabla\Delta u|+|\Delta^2d|)dx,
\end{eqnarray*}
for a positive constant $C$ depending only $\bar\rho$. Using the H\"older and Sobolev inequalities, as well as the Gagliardo-Nirenberg inequality, $\|f\|_{L^\infty(\mathbb R^3)}\leq C\|f\|_{L^6(\mathbb R^3)}^{\frac12}\|\Delta f\|_{L^2(\mathbb R^3)}^{\frac12}$, it follows from the above inequality and (\ref{NNL2}) that
\begin{eqnarray*}
  &&\int_{\mathbb R^3}(|\nabla\Delta u|^2+\varepsilon|\Delta^2u|^2+|\Delta^2d|^2)dx\nonumber \\
  &\leq&\|\nabla\rho\|_{L^3}(\|\partial_tu\|_{L^6}+\|u\|_{L^\infty}\|\nabla u\|_{L^6})\|\nabla\Delta u\|_{L^2}\\
  &&+C[\|\nabla\partial_tu\|_{L^2}+\|\Delta\partial_td\|_{L^2}+(\|u\|_{L^6}+\|\nabla d\|_{L^6})(\|\nabla^2u\|_{L^3}+\|\nabla^3d\|_{L^3})\\
  &&+\|\nabla u\|_{L^4}^2+\|\nabla^2d\|_{L^4}^2](\|\nabla\Delta u\|_{L^2}+\|\Delta^2d\|_{L^2})\\
  &\leq&C\|\nabla\rho\|_{L^3}(\|\nabla\partial_tu\|_{L^2}+\|\nabla u\|_{L^2}^{\frac12}\|\nabla^2u\|_{L^2}^{\frac12}\|\nabla^2u\|_{L^2})\|\nabla\Delta u\|_{L^2}+C[\|\nabla\partial_tu\|_{L^2}\\
  &&+\|\Delta\partial_td\|_{L^2}+(\|\nabla u\|_{L^2}+\|\nabla^2d\|_{L^2}) (\|\nabla^2u\|_{L^2}^{\frac12}+\|\nabla^3d\|_{L^2}^{\frac12} )(\|\nabla\Delta u\|_{L^2}^{\frac12} \\
  && +\|\Delta^2d\|_{L^2}^{\frac12})+ \|\nabla u\|_{L^2}^{\frac12}\|\nabla^2u\|_{L^2}^{\frac32}+\|\nabla^2d \|_{L^2}^{\frac12}\|\nabla\Delta d\|_{L^2}^{\frac32}](\|\nabla\Delta u\|_{L^2} +\|\Delta^2d\|_{L^2})\\
  &\leq&C\|\nabla\rho\|_{L^3}(\|\nabla\partial_tu\|_{L^2}+1)\|\nabla\Delta u\|_{L^2}+C(\|\nabla\partial_tu\|_{L^2}+\|\Delta\partial_td\|_{L^2}\\
  &&+\|\nabla\Delta u\|_{L^2}^{\frac12}+\|\Delta^2d\|_{L^2}^{\frac12})(\|\nabla\Delta u\|_{L^2}+\|\Delta^2d\|_{L^2})\\
  &\leq&\frac12(\|\nabla\Delta u\|_{L^2}^2+\|\Delta^2d\|_{L^2}^2)+C\|\nabla\rho\|_{L^3}^2\|\nabla\partial_t u\|_{L^2}^2\\
  &&+C(\|\nabla\partial_tu\|_{L^2}^2+\|\Delta\partial_td\|_{L^2}^2 +1),
\end{eqnarray*}
for a positive constant $C$ depending only on $\bar\rho$, $\|\nabla u_0\|_{H^1}+\|\nabla d_0\|_{H^2}$, $\|g_0\|_{L^2}$ and $\varepsilon \left\|\frac{\Delta^2u_0}{\sqrt{\rho_0}} \right\|_{L^2}$.
Therefore, we have
\begin{align}
  \label{NL4}
  &\|\nabla\Delta u\|_{L^2}^2(t)+\varepsilon\|\Delta^2u\|_{L^2}^2(t)+\|\Delta^2d\|_{L^2}^2(t) \nonumber\\
  \leq& C\|\nabla\rho\|_{L^3}^2(t)\|\nabla\partial_t u\|_{L^2}^2(t)+C(\|\nabla\partial_tu\|_{L^2}^2(t) +\|\Delta\partial_td\|_{L^2}^2(t) +1),
\end{align}
for any $t\in(0,\min\{t_*,\tilde T\})$, where $C$ is a positive constant depending only on $\bar\rho$, $\|\nabla u_0\|_{H^1}+\|\nabla d_0\|_{H^2}$, $\|g_0\|_{L^2}$ and $\varepsilon \left\|\frac{\Delta^2u_0}{\sqrt{\rho_0}} \right\|_{L^2}$.

Applying the operator $\nabla$ to  equation $(\ref{EL})_1$, multiplying the resulting equation by $3|\n\rho|\n\rho$ and integrating over $\mathbb R^3$, it follows from integrating by parts, the Gagliardo-Nirenberg inequality, $\|f\|_{L^\infty(\mathbb R^3)}\leq C\|f\|_{L^6(\mathbb R^3)}^{\frac12}\|\Delta f\|_{L^2(\mathbb R^3)}^{\frac12}$, and the Young inequality
that
\begin{eqnarray*}
\frac{d}{dt}\|\n\rho\|^3_{L^3} &\leq& 3\|\n u\|_{L^{\infty}}\|\n\rho\|^3_{L^3}
\leq C\|\nabla u\|_{L^6}^{\frac12}\|\nabla\Delta u\|_{L^2}^{\frac12}\|\nabla\rho \|_{L^3}^3\\
&\leq&C\|\n ^2u\|^{\frac12}_{L^2}\|\n^3u\|^{\frac12}_{L^2}\|\n\rho\|^3_{L^3}.
\end{eqnarray*}
Thanks to (\ref{NL4}), it follows from the above inequality and the Young inequality that
$$
\frac{d}{dt}(1+\|\nabla \rho\|_{L^3}^3)\leq C(1+\|\nabla\partial_tu\|_{L^2}^2 +\|\Delta\partial_td\|_{L^2}^2)^{\frac14}(1+\|\nabla\rho\|_{L^3}^3)^2.
$$
Solving this ordinary differential inequality, and using the H\"older inequalty yields
\begin{eqnarray*}
  -\frac{1}{1+\|\nabla \rho\|_{L^3}^3}&\leq&C\int_0^t(1+\|\nabla\partial_tu\|_{L^2}^2 +\|\Delta\partial_td\|_{L^2}^2)^{\frac14}ds -\frac{1}{1+\|\nabla\rho_0\|_{L^3}^3} \\
  &\leq&C t^{\frac34} \left[\int_0^t(1+\|\nabla\partial_tu\|_{L^2}^2 +\|\Delta\partial_td\|_{L^2}^2)ds\right]^{\frac14} -\frac{1}{1+\|\nabla\rho_0\|_{L^3}^3}\\
  &\leq&C_{**}t^{\frac34}-\frac{1}{1+\|\nabla\rho_0\|_{L^3}^3}\leq -\frac{1}{2(1+\|\nabla\rho_0\|_{L^3}^3)},
\end{eqnarray*}
for any $t\leq t_{**}:=\min\left\{ \left(\frac{1}{2C_{**}(1+\|\nabla\rho_0\|_{L^3}^3)}\right)^{\frac43},t_*, \tilde T\right\}$, where $C_{**}$ is a positive constant depending only on $\bar\rho$, $\|\nabla u_0\|_{H^1}+\|\nabla d_0\|_{H^2}$, $\|g_0\|_{L^2}$ and $\varepsilon \left\|\frac{\Delta^2u_0}{\sqrt{\rho_0}} \right\|_{L^2}$. The above inequality implies
\begin{equation}\label{eqeq1}
\|\nabla\rho\|_{L^3}^3\leq 1+2\|\nabla\rho_0\|_{L^3}^3,
\end{equation}
for any $t\leq t_{**}$. Thanks to this, by equation $(\ref{EL})_1$, and using the Gagliardo-Nirenberg and Sobolev
inequalities, we deduce
\begin{eqnarray*}
\|\partial_t\rho\|_{L^3}&\leq&\|u\|_{L^\infty}\|\nabla\rho\|_{L^3}\leq C\|u\|_{L^6}^{\frac12}\|\Delta u\|_{L^2}^{\frac12}\|\nabla\rho\|_{L^3}\\
&\leq&C\|\nabla u\|_{L^2}^{\frac12}\|\Delta u\|_{L^2}^{\frac12}\|\nabla\rho\|_{L^3}.
\end{eqnarray*}
Combining this with (\ref{NNL2}) and (\ref{eqeq1}), one obtains the desired a priori estimate stated in Proposition \ref{LEM5}, on the time interval $(0,t_{**})$. While this a priori estimate in turn, by the local existence, i.e.\,Lemma \ref{ELLOC}, implies that $t_{**}\leq\tilde T$, otherwise, one can extend $(\rho, u, d)$ beyond the time $\tilde T$, which contradicts to the definition of $\tilde T$.  This proves the conclusion.
\end{proof}

\section{Proof of Theorem $\ref{th1}$}
\label{sec3}
Now we can prove the local existence and uniqueness of strong solutions to (\ref{el}).
\begin{proof}[Proof of Theorem \ref{th1}]
\textbf{Existence.} Let $j_\varepsilon$ be a standard modifier, that is $j_\varepsilon(x)=\frac{1}{\varepsilon^3}j\left(\frac{x}{\varepsilon}\right)$, with $0\leq j\in C_0^\infty(\mathbb R^3)$ and $\int_{\mathbb R^3}j(x)dx=1$. For any $\varepsilon>0$, we set
$$
u_{0\varepsilon}=j_{\varepsilon^{1/4}}*u_0,\quad d_{0\varepsilon}=d_0,\quad \rho_{0\varepsilon}=\rho_0+\varepsilon+\delta_\varepsilon,\quad \mbox{with } \delta_\varepsilon=\|u_{0\varepsilon}-u_0\|_{D^{2,2}},
$$
where $j_{\varepsilon^{1/4}}*u_0$ denotes the convolution of $j_{\varepsilon^{1/4}}$ with $u_0$, and set
$$
g_{0\varepsilon}=\frac{1}{\sqrt{\rho_{0\varepsilon}}}[\Delta u_{0\varepsilon}+\nabla P_0-\text{div}(\nabla d_{0\varepsilon}\odot\nabla d_{0\varepsilon}+(\Delta d_{0\varepsilon}+|\nabla d_{0\varepsilon}|^2d_{0\varepsilon})\otimes d_{0\varepsilon}].
$$
By the properties of the convolution, one has $u_{0\varepsilon}\rightarrow u_0$ in $D^{2,2}(\mathbb R^3)$, that is $\delta_\varepsilon\rightarrow0$. Noticing that
$g_{0\varepsilon}=\frac{1}{\sqrt{\rho_{0\varepsilon}}} (\sqrt{\rho_0}g_0+\Delta(u_{0\varepsilon} -u_0)),$ $\rho_0\leq\rho_{0\varepsilon}$ and $\rho_{0\varepsilon}\geq \delta_{0\varepsilon}$,
we have
$$
\|g_{0\varepsilon}\|_{L^2}\leq \|g_0\|_{L^2}+\sqrt{\delta_\varepsilon} \leq \|g_0\|_{L^2}+1,
$$
for $\varepsilon$ sufficiently small. Simple calculations yield
$$
\Delta^2u_{0\varepsilon}=\frac{1}{\sqrt\varepsilon}\int_{\mathbb R^3}\frac{1}{\varepsilon^{3/4}}\Delta j\left(\frac{x-y}{\varepsilon^{1/4}}\right)\Delta u_0(y)dy=\frac{1}{\sqrt\varepsilon} (\Delta j)_{\varepsilon^{1/4} }*\Delta u_0,
$$
where $(\Delta j)_{\varepsilon^{1/4}}(x)=\frac{1}{\varepsilon^{3/4}}\Delta j\left(\frac{x}{\varepsilon^{1/4}}\right)$, hence, noticing that $\Delta j\in C_0^\infty(\mathbb R^3)$ and $\sqrt{\rho_{0\varepsilon}}\geq\sqrt\varepsilon$, by the properties of convolution, we have
$$
\varepsilon\left\|\frac{\Delta^2u_{0\varepsilon}}{\sqrt{\rho_{0\varepsilon}}}
\right\|_{L^2}\leq\sqrt\varepsilon\|\Delta^2u_{0\varepsilon}\|_{L^2}
=\|(\Delta j)_{\varepsilon^{1/4}}*\Delta u_0\|_{L^2}\leq C\|\Delta u_0\|_{L^2},
$$
for a positive constant $C$ independent of $\varepsilon$. Therefore, we have
\begin{equation*}
\|\nabla u_{0\varepsilon}\|_{H^1}+\|\nabla d_{0\varepsilon}\|_{H^2}+\|\nabla\rho_{0\varepsilon} \|_{L^3}+ \varepsilon\left\|\frac{\Delta^2u_{0\varepsilon}}{\sqrt{\rho_{0\varepsilon}}}
\right\|_{L^2}+\|g_{0\varepsilon}\|_{L^2}\leq C,
\end{equation*}
for a positive constant $C$ depending only on $\|\nabla u_0\|_{H^1}+\|\nabla d_0\|_{H^2}+\|\nabla\rho_0\|_{L^3}+\|g_0\|_{L^2}$.

For any $\varepsilon>0$, by Lemma \ref{ELLOC} and Proposition \ref{LEM5}, there is a positive time $T^*>0$, depending only on $\bar\rho$ and $\|\nabla u_0\|_{H^1}+\|\nabla d_0\|_{H^2}+\|\nabla\rho_0\|_{L^3}+\|g_0\|_{L^2}$, such that system (\ref{EL}) has a unique strong solution $(\rho_\varepsilon, u_\varepsilon,d_\varepsilon)$, on $\mathbb R^3\times(0,T_*)$, with initial data $(\rho_{0\varepsilon}, u_{0\varepsilon}, d_{0\varepsilon})$, satisfying the estimate
\begin{align*}
  &\sup_{0\leq t\leq T_*}(\| \nabla u_\varepsilon\|_{H^1}^2+\|\nabla d_\varepsilon\|_{H^2}^2+\|\nabla\partial_td_\varepsilon\|_{L^2}^2 +\|\nabla\rho_\varepsilon \|_{L^3}^3+\|\partial_t\rho_\varepsilon\|_{L^3}^3) \\
  &+\int_0^{ T_*}(\|\nabla\partial_tu_\varepsilon\|_{L^2}^2+\|\nabla\Delta u_\varepsilon\|_{L^2}^2+\|\Delta\partial_td_\varepsilon\|_{L^2}^2+\|\Delta^2d_\varepsilon \|_{L^2}^2ds)\leq C,
\end{align*}
for a positive constant $C$ depending only on $\bar\rho$ and $\|\nabla u_0\|_{H^1}+\|\nabla d_0\|_{H^2}+\|\nabla\rho_0\|_{L^3}+\|g_0\|_{L^2}$, and in particular independent of $\varepsilon$. Thanks to this a priori estimate, it is then standard to show the existence of strong solution $(\rho, u, d)$ to system (\ref{el}), subject to (\ref{bc})--(\ref{ic}), by studying the limit $\varepsilon\rightarrow0^+$, and using the Aubin-Lions lemma and Cantor's diagonal argument.

\textbf{Uniqueness.}
We first note that, by the regularities of the strong solution $(\rho, u, d)$, one has $\nabla u\in L^1(0,T_*; W^{1,\infty})$, therefore, by the assumption that $\rho_0\in L^{\frac32}$ and $\nabla\rho_0\in L^2$, it follows from the transport equation $(\ref{el})_1$ that $\rho$
has the following additional regularities:
$$
\rho\in L^\infty(0,T_*; L^{\frac32}),\quad\nabla\rho\in L^\infty(0,T_*; L^2).
$$

Let $(\rho_i,u_i,d_i), i=1,2$ be two strong solutions to system (\ref{el}), on $\mathbb{R}^3\times(0,T_*)$, with $(\rho_i, u_i, d_i)|_{t=0}=(\rho_0, u_0, d_0)$. Set $\tilde{u}=u_1-u_2, \tilde{d}=d_1-d_2$ and $\tilde{\rho}=\rho_1-\rho_2$.
For convenience, we denote
$$
\mathscr S_i=(\Delta d_i+|\nabla d_i|^2d_i)\otimes d_i,\quad \mathscr Q_i=(d_i\cdot\nabla)u_i-(d_i^TA_id_i)d_i,\quad i=1,2.
$$
Recalling the equalities (\ref{NL1}) and (\ref{NL2}), it is clear that
$$
\mathscr S_i=[(d_i\times\Delta d_i)\times d_i]\otimes d_i,\quad\mathscr Q_i=[d_i\times(d_i\cdot\nabla)u_i]\times d_i.
$$
Then, straightforward calculations yield
\begin{eqnarray}
  \mathscr S_1-\mathscr S_2=[(d_1\times\Delta \tilde d)\times d_1]\otimes d_1+\mathscr R_s,\quad\mbox{with }|\mathscr R_s|\leq 3|\Delta d_1||\tilde d|,\label{NNNL1}\\
  \mathscr Q_1-\mathscr Q_2=[d_1\times(d_1\cdot\nabla)\tilde u]\times d_1+\mathscr R_q, \quad\mbox{with }|\mathscr R_q|\leq3|\nabla u_1||\tilde d|. \label{NNNL2}
\end{eqnarray}
Subtracting the equations for $(\rho_1, u_1, d_1)$ from those for $(\rho_2, u_2, d_2)$, one obtains the following system for $(\tilde\rho, \tilde u,\tilde d)$:
\begin{equation}\label{elun}\left\{
\begin{array}{l}
\partial_t\tilde{\rho}+u_1\cdot\n\tilde{\rho}+\tilde u\cdot\n\rho_2=0,\\
\rho_1(\partial_t\tilde{u}+(u_1\cdot\n)\tilde{u}) +\tilde{\rho}(\partial_tu_{2}+(u_2\cdot\n) u_2)+\rho_1(\tilde{u}\cdot\n) u_2+\n\tilde {P}\\
\qquad\ =\Delta\tilde{u}-\text{div}(\n \tilde{d}\odot\n d_1+\n d_2\odot\n\tilde{d})-\text{div}(\mathscr S_1-\mathscr S_2),\\
\text{div}\tilde{u}=0,\\
\partial_t\tilde{d}+(\tilde{u}\cdot \n) d_1+(u_2\cdot \n)\tilde{d}\\
\qquad\ =\Delta \tilde{d}+\n\tilde{d}:(\n d_1+\n d_2)d_1+|\n d_2|^2\tilde{d}+\mathscr Q_1-\mathscr Q_2.
\end{array}
\right.
\end{equation}

Multiplying $(\ref{elun})_2$ by $\tilde{u} $ and $(\ref{elun})_4$ by $-\Delta\bar{ d}$, respectively, summing the resultants up and integrating over $\mathbb{R}^3$, it follows from integration by parts that
\begin{align}\label{unn1}
&\frac12\frac {d}{dt}\int_{\mathbb{R}^3}(\rho_1|\tilde{u}|^2+|\n \tilde{d}|^2)dx+\int_{\mathbb{R}^3}(|\n \tilde{u}|^2+|\Delta\tilde{d}|^2)dx\nonumber\\
=&\int_{\mathbb{R}^3}\big\{-[\tilde{\rho}(\partial_tu_{2}+u_2\cdot\n u_2)+\rho_1\tilde{u}\cdot \n u_2]\cdot\tilde{u}+(\n\tilde{d}\odot\n d_1+\n d_2\odot \n \tilde{d}):\n\tilde{u}\big\}dx\nonumber\\
&+\int_{\mathbb{R}^3}\big\{(u_2\cdot\n)\tilde{d}\cdot\Delta\tilde{d}-((\partial_i\tilde{u}\cdot\n) d_1+(\tilde u\cdot\nabla)\partial_id_1)\cdot\partial_i\tilde d-[\n\tilde{d}:\n(d_1+d_2)d_1\nonumber\\
&+|\n d_2|^2\tilde{d}]\cdot \Delta\tilde{d}\big\}dx+\int_{\mathbb R^3}[(\mathscr S_1-\mathscr S_2):\nabla\tilde u-(\mathscr Q_1-\mathscr Q_2)\cdot\Delta\tilde d]dx.
\end{align}
Thanks to (\ref{NNNL1}) and (\ref{NNNL2}), and using the identity $(a\times b)\cdot c=(b\times c)\cdot a$ twice,
one obtains
\begin{eqnarray*}
  &&\int_{\mathbb R^3}[(\mathscr S_1-\mathscr S_2):\nabla\tilde u-(\mathscr Q_1-\mathscr Q_2)\cdot\Delta\tilde d]dx\\
  &=&\int_{\mathbb R^3}\big\{[(d_1\times\Delta \tilde d)\times d_1]\otimes d_1:\nabla\tilde u  -[d_1\times(d_1\cdot\nabla)\tilde u]\times d_1\cdot\Delta\tilde d\big\}dx\\
  &&+\int_{\mathbb R^3}(\mathscr R_s :\nabla\tilde u-\mathscr R_q\cdot\Delta\tilde d)dx\\
  &=&\int_{\mathbb R^3}\big\{(d_1\cdot\nabla)\tilde u \cdot [(d_1\times\Delta \tilde d)\times d_1]  -[d_1\times(d_1\cdot\nabla)\tilde u]\times d_1\cdot\Delta\tilde d\big\}dx\\
  &&+\int_{\mathbb R^3}(\mathscr R_s :\nabla\tilde u-\mathscr R_q\cdot\Delta\tilde d)dx=\int_{\mathbb R^3}(\mathscr R_s :\nabla\tilde u-\mathscr R_q\cdot\Delta\tilde d)dx\\
  &\leq&3\int_{\mathbb R^3}(|\Delta d_1||\tilde d||\nabla\tilde u|+|\nabla u_1||\tilde d||\Delta\tilde d|)dx.
\end{eqnarray*}
Substituting the above inequality into (\ref{unn1}), and noticing that the regularities of $(\rho_i, u_i, d_i)$ imply
$$
\sup_{0\leq t\leq T_*}(\|(u_i, \nabla d_i)\|_{L^6}+\|(\nabla u_i, \nabla^2d_i)\|_{L^3}+\|\nabla d_i\|_{L^\infty})\leq C,
$$
it follows from the H\"older, Sobolev and Young inequalities that
\begin{eqnarray*}
&&\frac12\frac {d}{dt}\int_{\mathbb{R}^3}(\rho_1|\tilde{u}|^2+|\n \tilde{d}|^2dx)+\int_{\mathbb{R}^3}(|\n \tilde{u}|^2+|\Delta\tilde{d}|^2)dx\nonumber\\
&\leq&\int_{\mathbb R^3}\big\{[|\tilde\rho|(|\partial_tu_2|+|u_2||\nabla u_2|)+\rho_1|\tilde u||\nabla u_2|]|\tilde u| +(|\nabla d_1|+|\nabla d_2|)|\nabla\tilde d||\nabla\tilde u|\nonumber\\
&&+|u_2||\nabla\tilde d||\Delta\tilde d|+(|\nabla\tilde u||\nabla d_1|+|\tilde u||\nabla^2d_1|)|\nabla\tilde d|
+[|\nabla\tilde d|(|\nabla d_1|+|\nabla d_2|)\nonumber\\
&&+|\nabla d_2|^2|\tilde d|] |\Delta\tilde d|\big\}dx+3\int_{\mathbb R^3}(|\Delta d_1||\tilde d||\nabla\tilde u|+|\nabla u_1||\tilde d||\Delta\tilde d|)dx\nonumber\\
&\leq&\|\tilde{\rho}\|_{L^{\frac32}}[(\|\partial_tu_{2}\|_{L^6}+\|u_2\|_{L^6}\|\n u_2\|_{L^\infty})\|\tilde{u}\|_{L^6}+\|\n u_2\|_{L^{\infty}}\|\sqrt{\rho_1}\tilde{u}\|^2_{L^2}]\nonumber\\
&&+(\|\n d_1\|_{L^{\infty}}+\|\n d_2\|_{L^{\infty}})\|\n \tilde{d}\|_{L^2}\|\n\tilde{u}\|_{L^2}+\|u_2\|_{L^6}\|\nabla\tilde d\|_{L^3}\|\Delta\tilde d\|_{L^2}\nonumber\\
&&+\|\nabla\tilde u\|_{L^2}\|\nabla d_1\|_{L^\infty}\|\nabla\tilde d\|_{L^2}+\|\tilde{u}\|_{L^6}\|\n^2 d_1\|_{L^3}\|\n\tilde{d}\|_{L^2}\nonumber\\
&&+\|\n\tilde{d}\|_{L^2}(\|\n d_1\|_{L^{\infty}}+\|\n d_2\|_{L^{\infty}})\|\Delta\tilde{d}\|_{L^2}+\|\n d_2\|^2_{L^6}\|\tilde{d}\|_{L^6}\|\Delta\tilde{d}\|_{L^2}\nonumber\\
&&+3\|\Delta d_1\|_{L^3}\|\tilde{d}\|_{L^6} |\n\tilde{u}\|_{L^2}+3\|\n u_1\|_{L^3}\|\tilde{d}\|_{L^6}\|\Delta\tilde{d}\|_{L^2}\nonumber\\
&\leq&C\Big\{\|\tilde\rho\|_{L^{\frac32}}[(\|\nabla\partial_tu_2\|_{L^2}+\|\nabla u_2\|_{H^2})\|\nabla\tilde u\|_{L^2}+\|\nabla u_2\|_{H^2} \|\sqrt{\rho_1}\tilde u\|_{L^2}]\nonumber\\
&&+\|\nabla\tilde d\|_{L^2}\|\nabla\tilde u\|_{L^2}+\|\nabla\tilde d\|_{L^2}^{\frac12}\|\Delta\tilde d\|_{L^2}^{\frac32}+\|\nabla\tilde d\|_{L^2}\|\Delta\tilde d\|_{L^2}\Big\}\\
&\leq&C(\|\nabla\partial_t u_2\|_{L^2}^2+\|\nabla u_2\|_{H^2}^2+1)(\|\sqrt\rho_1\tilde{u}\|^2_{L^2}+\|\tilde{\rho}\|^2_{L^{\frac32}}+\|\n\tilde{d}\|^2_{L^2})\nonumber\\
&&+\frac12(\|\n\tilde{u}\|^2_{L^2}+\|\Delta\tilde{d}\|^2_{L^2}).
\end{eqnarray*}
Hence
\begin{eqnarray}
  &&\frac{d}{dt}(\|\sqrt\rho_1\tilde{u}\|^2_{L^2}+\|\n\tilde{d}\|^2_{L^2})+ (\|\n\tilde{u}\|^2_{L^2}+\|\Delta\tilde{d}\|^2_{L^2})\nonumber\\
  &\leq&C(\|\nabla\partial_t u_2\|_{L^2}^2+\|\nabla u_2\|_{H^2}^2+1)(\|\sqrt\rho_1\tilde{u}\|^2_{L^2}+\|\tilde{\rho}\|^2_{L^{\frac32}}+\|\n\tilde{d}\|^2_{L^2}). \label{NNNL3}
\end{eqnarray}

Multiplying $(\ref{elun})_1$ by $\frac32|\tilde{\rho}|^{-\frac12}\tilde \rho$, it follows from integration by parts, the H\"older and Sobolev inequalities that
\begin{eqnarray*}
\frac{d}{dt} \|\tilde{\rho}\|_{L^{\frac32}}^{\frac32}
 &\leq& \frac32\int_{\mathbb{R}^3}|\tilde{\rho}|^{\frac12}|\tilde{u}||\n\rho_2|dx \leq \frac32\|\tilde{\rho}\|^{\frac12}_{L^{\frac32}}\|\n \rho_2\|_{L^2}\|\tilde{u}\|_{L^6}\nonumber\\
&\leq&C\|\tilde{\rho}\|^{\frac12}_{L^{\frac32}}\|\n \rho_2\|_{L^2}\|\n\tilde{u}\|_{L^2}.
\end{eqnarray*}
Multiplying the above inequality by $\|\tilde{\rho}\|^{\frac12}_{L^{\frac32}}$, and using Young's inequality, we obtain
\begin{eqnarray}\label{urholast}
&&\frac{d}{dt}\|\tilde{\rho}\|^2_{L^{\frac32}}\leq C\|\tilde{\rho}\|_{L^{\frac{3}{2}}}\|\n\tilde{u}\|_{L^2}\leq\frac12\|\n\tilde{u}\|^2_{L^2}+C\|\nabla\rho_2\|_{L^2}^2 \|\tilde{\rho}\|^2_{L^{\frac32}}.
\end{eqnarray}

Summing (\ref{NNNL3}) with (\ref{urholast}), one obtains
\begin{align*}
&\frac {d}{dt}(\|\sqrt\rho_1\tilde{u}\|^2_{L^2}+\|\n\tilde{d}\|^2_{L^2}+\|\tilde{\rho}\|^2_{L^{\frac32}})+ \frac12(\|\n\tilde{u}\|^2_{L^2}+\|\Delta\tilde{d}\|^2_{L^2})\\
 \leq &C(\|\nabla\partial_t u_2\|_{L^2}^2+\|\nabla u_2\|_{H^2}^2+\|\nabla\rho_2\|_{L^2}^2+1)(\|\sqrt\rho_1\tilde{u}\|^2_{L^2}+\|\n\tilde{d}\|^2_{L^2}+\|\tilde{\rho}\|^2_{L^{\frac32}}),
\end{align*}
from which, recalling that $(\tilde\rho, \tilde u, \tilde d)|_{t=0}=(0,0,0)$, by the Gronwall inequality, one has
$(\tilde{\rho},\tilde{u},\n\tilde{d})\equiv0.$ This, by the Sobolev embedding inequality, further implies $(\tilde\rho, \tilde u,\tilde d)\equiv0$. The proof is complete.
\end{proof}

\section*{Acknowledgments}
The authors are grateful to Prof. Xiangao Liu and Dr. Jinrui Huang for their helpful comments. Huajun Gong is partial supported by Posted-doctor Foundation of China 2015M580728 and Shenzhen University 2014-62 and NSFC No.61401283. Chen Xu is partial supported by Guangdong Provincial Science and Technology Plan Project No.2013B040403005, No.GCZX-A1409 and NSFC No.61472257.

\par

\end{document}